\documentclass[preprint,12pt,authoryear]{elsarticle}

\usepackage{tabularx} 
\usepackage{enumitem}
\usepackage{amsmath,amssymb,amsfonts}

\usepackage{longtable}
\usepackage{booktabs}
\usepackage{graphicx} 
\usepackage{subcaption} 
\usepackage{algorithm}
\usepackage{algpseudocode}
\usepackage{lscape} 
\usepackage[utf8]{inputenc}
\usepackage{booktabs}  
\usepackage{tablefootnote}
\usepackage{xcolor}
\usepackage[bookmarks=false]{hyperref}
\usepackage{amssymb}
\usepackage{amsmath}
\newtheorem{theorem}{Theorem}[section]
\newtheorem{lemma}{Lemma}[section]
\newtheorem{proof}{Proof}[section]

\journal{Elsevier}

\begin{document}

\begin{frontmatter}




\title{\texorpdfstring{$PG-NODE^{TB}$}{PG-NODE\textsuperscript{TB}}: Physics-Guided Neural Ordinary Differential Equations for Tuberculosis Transmission Dynamics}


\author[a,b]{Selain K. Kasereka}
\author[c]{Eric M. Mafuta}
\author[d]{Fadi Al Machot}
\author[e]{Emmanuel M. Kabengele}
\author[a]{Jean Chamberlain Chedjou}
\author{Kyandoghere Kyamakya\corref{cor1}\fnref{a,b}}
 \ead{kyandoghere.kyamakya@aau.at}
 \cortext[cor1]{Corresponding author:}

\address[a]{Institute of Smart Systems Technologies, University of Klagenfurt, Klagenfurt, Austria}
\address[b]{ABIL-LAB, University of Kinshasa, Kinshasa, Democratic Republic of the Congo}
\address[c]{School of Public Health, Faculty of Medicine, University of Kinshasa, Kinshasa, Democratic Republic of the Congo} 
\address[d]{Department of Data Science, Norwegian University of Life Sciences, Norway}
\address[e]{Institute of Global Health, Faculty of Medicine, University of Geneva, Geneva, Switzerland}

\begin{abstract}
Tuberculosis (TB) remains a leading global infectious disease, causing approximately 1.3 million deaths and 10.6 million new infections annually. Classical compartmental ODE models are the standard epidemiological tool for TB, yet their fixed-parameter structure cannot adapt to time-varying dynamics, unmodeled effects, or heterogeneous real-world data. This paper presents a methodological framework and proof-of-concept for applying Physics-Guided Neural Ordinary Differential Equations (PG-NODE) to TB transmission modeling within a SLIR (Susceptible, Latent, Infectious, Recovered) compartmental framework. We perform a rigorous mathematical analysis of the SLIR model, including derivation of the basic reproduction number $\mathcal{R}_0$, equilibrium analysis, and normalized sensitivity indices. We then reformulate the SLIR system as a PG-NODE, preserving compartmental conservation laws and biological constraints while enabling neural network components to learn unknown or time-varying rate functions from data. Three simulation scenarios illustrate the framework's intended capabilities: (i) adaptive tracking of time-varying transmission rates, (ii) correcting for unmodeled treatment and relapse dynamics with 27\% lower RMSE than the classical SLIR, and (iii) comparative forecasting of competing intervention policies over a 20-year horizon. Simulation results indicate that PG-NODE has strong potential for improving predictive accuracy while maintaining epidemiological interpretability; full adjoint-based training on real WHO surveillance data is identified as the key next step for empirical validation.
\end{abstract}

\begin{graphicalabstract}
\includegraphics[width=\linewidth]{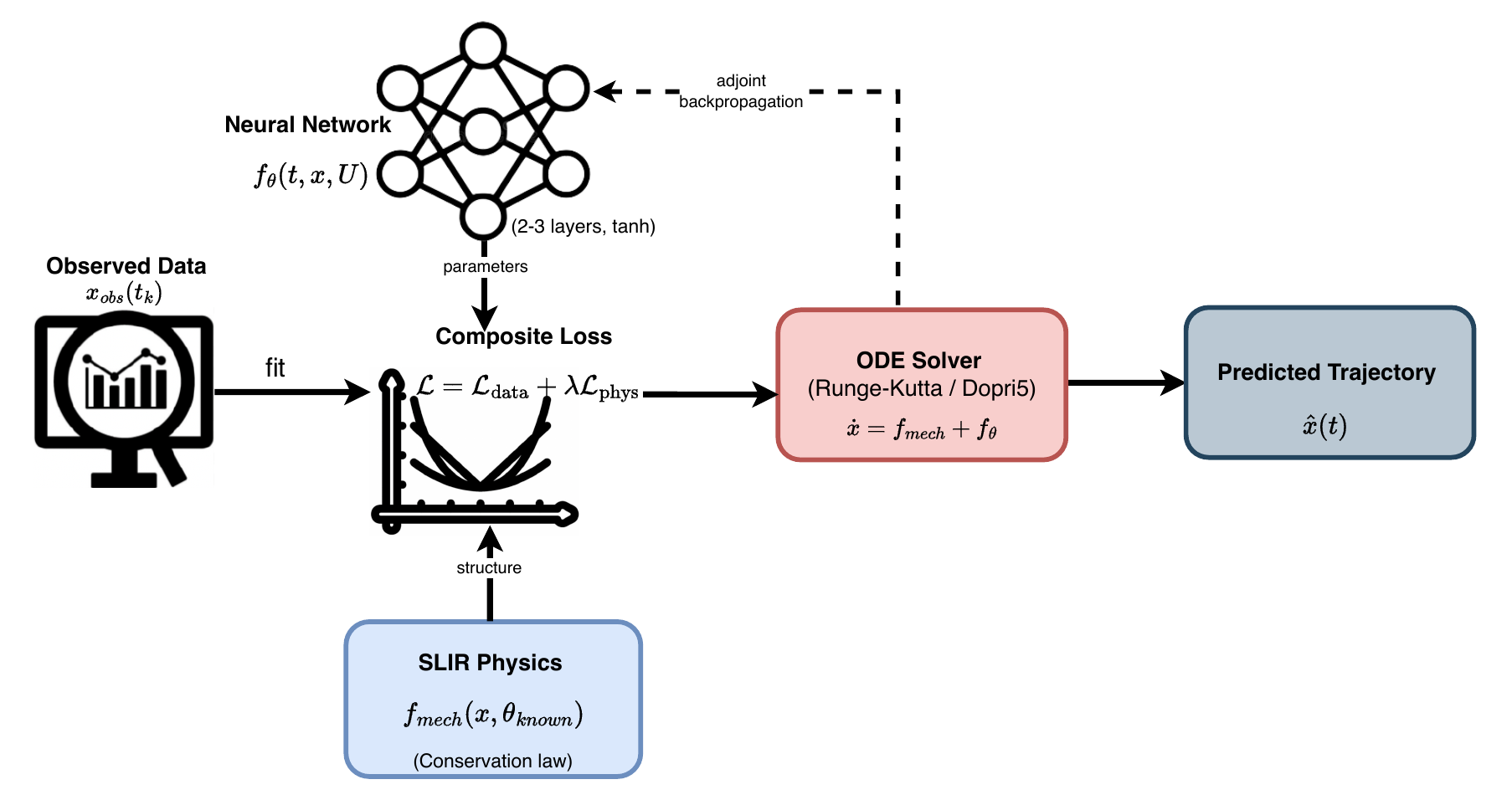}
\end{graphicalabstract}

\begin{highlights}
\item A PG-NODE framework is proposed that embeds SLIR mechanistic structure into neural ODEs while preserving biological and conservation constraints.
\item Rigorous mathematical analysis of the SLIR model provides closed-form equilibrium conditions, stability criteria, and intervention sensitivity rankings.
\item Treatment scale-up is identified as the most impactful single intervention for reducing TB transmission at the population level.
\item A learnable neural correction term compensates for structural model misspecification without requiring additional explicit compartments.
\item PG-NODE adaptively tracks non-stationary transmission dynamics driven by public health interventions, unlike fixed-parameter ODE models.
\item A PG-NODE-guided combined intervention strategy achieves superior long-term epidemic control and projects a trajectory toward TB elimination.
\end{highlights}


\end{frontmatter}



\section{Introduction \label{sec:into}}

Tuberculosis (TB), caused by Mycobacterium tuberculosis, remains one of the most consequential infectious diseases of the modern era. The World Health Organization reports that in 2024 TB caused an estimated 1.3 million deaths and 10.6 million new infections, with the greatest burden in South-East Asia (45\%), Africa (23\%), and the Western Pacific (18\%) \cite{WHO2025TB}. Despite effective chemotherapy since the 1950s and the BCG vaccine, global elimination targets set for 2035 remain far from reach, partly due to multi-drug-resistant TB (MDR-TB) and the TB/HIV syndemic \cite{uplekar2015s}.

Mathematical compartmental models have been central to understanding TB transmission dynamics since Waaler et al.\ \cite{Waaler1962}. The Kermack-McKendrick framework \cite{Kermack1927} gave rise to SEIR-type models that became the workhorses of TB modeling \cite{Blower1995,Anderson1991,ochieng2025mathematical,hattamurrahman2026mathematical}. Subsequent decades saw rapid elaboration: fast and slow latency stratification \cite{Feng2000,kasereka2020analysis}, explicit treatment and MDR-TB dynamics \cite{Castillo2002,Cohen2004}, behavioral and stigma parameters \cite{chikovore2017tb,kabunga2020stochastic}, age structure \cite{Hethcote2000}, and co-infection with HIV \cite{Roeger2009,raza2025mathematical}. Yet, despite this richness, all classical ODE models share a structural limitation: parameters are assumed time-invariant, estimated offline from aggregate data, and cannot adapt to non-stationary dynamics, unmodeled heterogeneity, or real-time surveillance streams.

Physics-Guided Neural Ordinary Differential Equations (PG-NODE) \cite{Chen2018} address these limitations by embedding mechanistic ODE structure into a neural network architecture. The key principle is hybridization: known epidemiological equations are preserved explicitly, while unknown, time-varying, or misspecified components are learned from data. This framework is particularly timely in the context of mobile and pervasive computing, where connected health systems generate continuous surveillance data streams that could drive real-time adaptive epidemic models.

This paper applies PG-NODE to a SLIR TB transmission model as a proof-of-concept methodological study. Our contributions are: (1) a rigorous mathematical analysis of the SLIR model including $\mathcal{R}_0$ derivation and sensitivity analysis; (2) a PG-NODE formulation preserving SLIR physics with a concrete training objective and architecture specification; (3) three simulation scenarios illustrating the framework's comparative advantages over classical SLIR; and (4) a discussion of current limitations, including the absence of real-data training, and a roadmap for empirical validation.

The remainder is organized as follows. Section~\ref{sec:methodo} presents the methodology. Section~\ref{sec:simulation} reports the simulations. Section~\ref{sec:discussion} discusses the findings, and Section~\ref{sec:conclusion} concludes.

\section{Methodology \label{sec:methodo}}

\subsection{Model Description}

We consider a compartmental SLIR model for tuberculosis in a closed population with vital dynamics. The total population at time $t$ is
\begin{equation}
    N(t) = S(t) + L(t) + I(t) + R(t),
    \label{eq:TotalPop}
\end{equation}
where $S(t)$, $L(t)$, $I(t)$, and $R(t)$ denote the susceptible, latently infected, infectious, and recovered subpopulations, respectively. Figure \ref{fig:DiaTB} illustrates the transmission dynamics. The governing ODE system is presented as shown in Eq. \eqref{eq:systemTB}.

\begin{figure*}[ht!]
\centering
    \includegraphics[scale=.5]{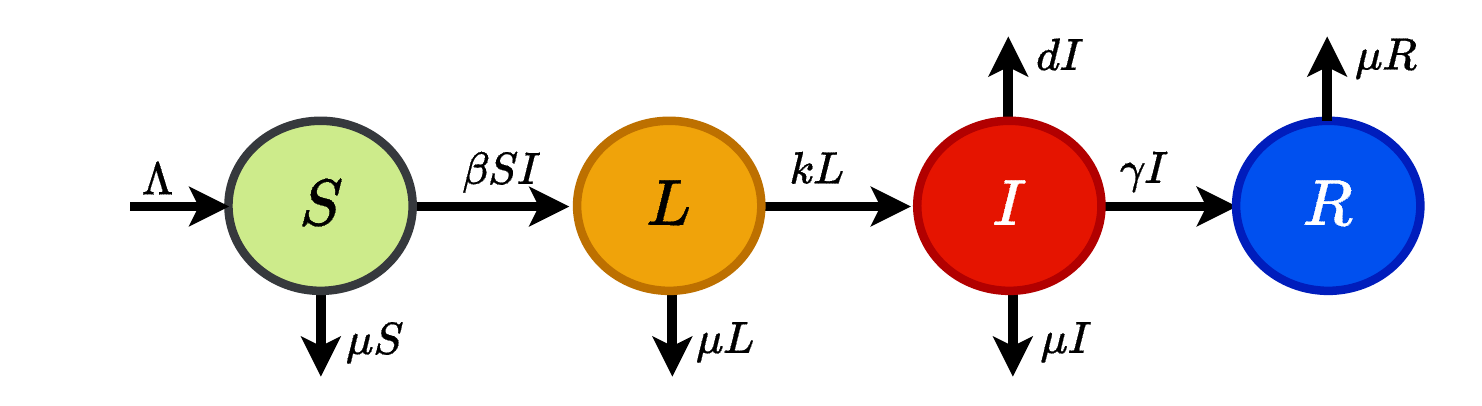}
    \caption{Diagram describing the dynamics between compartments for tuberculosis transmission.}
    \label{fig:DiaTB}
\end{figure*}

\begin{equation}
    \left\{
\begin{aligned}
\frac{dS}{dt} &= \Lambda - \beta \frac{S I}{N} - \mu S, \\[3pt]
\frac{dL}{dt} &= \beta \frac{S I}{N} - (k + \mu) L, \\[3pt]
\frac{dI}{dt} &= k L - (\gamma + \mu + d)\, I, \\[3pt]
\frac{dR}{dt} &= \gamma I - \mu R.
\end{aligned}
\right.
\label{eq:systemTB}
\end{equation}
The state variables and parameters are described below. $S(t)$ are susceptible individuals never infected with TB; $L(t)$ are latently infected (LTBI, non-infectious); $I(t)$ are active-TB infectious individuals; $R(t)$ are recovered individuals with partial immunity (no reinfection modeled). Parameters: $\Lambda$ is the per-capita recruitment (birth) rate; $\mu$ the natural mortality rate (applied to all compartments); $\beta$ the effective transmission rate; $k$ the progression rate from latency to active TB; $\gamma$ the recovery rate; $d$ the TB-induced additional mortality rate.

\subsection{Mathematical Analysis}

\subsubsection{Positivity and Forward Invariance}

\begin{lemma}
For any non-negative initial conditions $(S_0, L_0, I_0, R_0) \in \mathbb{R}^4_+$, the solution of system \eqref{eq:systemTB} remains in $\mathbb{R}^4_+$ for all $t \geq 0$, and the biologically feasible region
$\Omega = \{(S,L,I,R)\in\mathbb{R}^4_+\,:\, N \leq \Lambda/\mu\}$
is positively invariant.
\end{lemma}

\begin{proof}
On each boundary face, the vector field is inward or tangent: $\left.\frac{dS}{dt}\right|_{S=0}=\Lambda>0$; $\left.\frac{dL}{dt}\right|_{L=0}=\beta SI/N\geq 0$; $\left.\frac{dI}{dt}\right|_{I=0}=kL\geq 0$; $\left.\frac{dR}{dt}\right|_{R=0}=\gamma I\geq 0$. Hence solutions remain in $\mathbb{R}^4_+$. For the total population: $\dot{N}=\Lambda - \mu N - dI \leq \Lambda - \mu N$, so $\limsup_{t\to\infty} N(t) \leq \Lambda/\mu$. \qed
\end{proof}

\subsubsection{Existence and Uniqueness}

The right-hand side of \eqref{eq:systemTB} is a continuously differentiable ($C^1$) function on $\mathbb{R}_+\times\mathbb{R}^4$, hence locally Lipschitz. By the Picard-Lindelöf theorem, a unique solution exists locally. Global existence follows from the forward invariance of $\Omega$, which provides an a priori bound preventing finite-time blow-up.

\subsubsection{Equilibrium Analysis}

Setting all derivatives to zero in \eqref{eq:systemTB}:

\textbf{Disease-Free Equilibrium (DFE).} Setting $I=0$ yields
\begin{equation}
    \mathcal{E}_0 = \left(\frac{\Lambda}{\mu},\; 0,\; 0,\; 0\right).
    \label{eq:DFE}
\end{equation}

\textbf{Endemic Equilibrium (EE).} Let $\lambda^* = \beta I^*/N^*$ denote the endemic force of infection. Then:
\begin{align}
    S^* &= \frac{\Lambda}{\lambda^*+\mu},\quad L^* = \frac{\lambda^* S^*}{k+\mu},\quad I^* = \frac{kL^*}{\gamma+\mu+d},\quad R^* = \frac{\gamma I^*}{\mu}. \label{eq:EE}
\end{align}
Substituting and simplifying shows that $\lambda^* > 0$ (i.e., an EE exists) if and only if $\mathcal{R}_0 > 1$.

\subsubsection{Basic Reproduction Number $\mathcal{R}_0$}

Using the Next-Generation Matrix (NGM) method of van den Driessche and Watmough \cite{VandenDriessche2002}, with infected compartments $(L,\,I)$ evaluated at $\mathcal{E}_0$ (where $S^*/N^* = 1$):

\begin{equation}
    \mathbf{F} = \begin{pmatrix} 0 & \beta \\ 0 & 0 \end{pmatrix}, \quad
    \mathbf{V} = \begin{pmatrix} k+\mu & 0 \\ -k & \gamma+\mu+d \end{pmatrix}.
\end{equation}

Then $\mathbf{V}^{-1} = \dfrac{1}{(k+\mu)(\gamma+\mu+d)}\begin{pmatrix}\gamma+\mu+d & 0\\ k & k+\mu\end{pmatrix}$, and
\begin{equation}
    \mathcal{R}_0 = \rho(\mathbf{F}\mathbf{V}^{-1}) = \frac{\beta\, k}{(k+\mu)(\gamma+\mu+d)}.
    \label{eq:R0}
\end{equation}

$\mathcal{R}_0$ represents the expected number of secondary infections generated by one infectious individual in a fully susceptible population. With the baseline parameters of Table~\ref{tab:simparam}, $\mathcal{R}_0 \approx 3.61$, consistent with high-burden country estimates \cite{Blower1995}.

\subsubsection{Stability Analysis}

\begin{theorem}[Local Asymptotic Stability]
The DFE $\mathcal{E}_0$ is locally asymptotically stable (LAS) if\/ $\mathcal{R}_0 < 1$, and unstable if $\mathcal{R}_0 > 1$. A unique endemic equilibrium $\mathcal{E}^*$ exists and is LAS when $\mathcal{R}_0 > 1$.
\end{theorem}

\begin{proof}[Proof sketch]
The Jacobian of \eqref{eq:systemTB} at $\mathcal{E}_0$ has eigenvalues $-\mu$ (double, from $S$ and $R$ blocks) plus the eigenvalues of the $2\times 2$ sub-matrix $\mathbf{M} = \mathbf{F}-\mathbf{V}$:
$\mathrm{tr}(\mathbf{M}) = -(k+\mu) - (\gamma+\mu+d) < 0$ and
$\det(\mathbf{M}) = (k+\mu)(\gamma+\mu+d)(1-\mathcal{R}_0)$.
Hence $\det(\mathbf{M})>0$ iff $\mathcal{R}_0<1$, ensuring LAS. Instability for $\mathcal{R}_0>1$ and LAS of $\mathcal{E}^*$ follow by standard monotone system arguments \cite{Li1996}. \qed
\end{proof}

\subsubsection{Sensitivity Analysis}

The normalized sensitivity index of $\mathcal{R}_0$ with respect to parameter $\theta$ is $\Upsilon_\theta = (\partial\mathcal{R}_0/\partial\theta)\cdot(\theta/\mathcal{R}_0)$. The closed-form indices for system \eqref{eq:systemTB} are:

\begin{equation}
\Upsilon_\beta = +1,\quad
\Upsilon_k = \frac{\mu}{k+\mu},\quad
\Upsilon_\gamma = -\frac{\gamma}{\gamma+\mu+d},\quad
\Upsilon_d = -\frac{d}{\gamma+\mu+d},\quad
\Upsilon_\mu = -\frac{\mu}{k+\mu}-\frac{\mu}{\gamma+\mu+d}.
\end{equation}

With baseline parameters: $\Upsilon_\beta = +1.000$, $\Upsilon_k = +0.158$, $\Upsilon_\gamma = -0.858$, $\Upsilon_d = -0.129$, $\Upsilon_\mu = -0.171$. The transmission rate $\beta$ has the strongest positive impact ($\Upsilon_\beta = +1$), while the recovery rate $\gamma$ has the strongest protective effect ($\Upsilon_\gamma \approx -0.86$), confirming that treatment scale-up is the most efficient single intervention. The progression rate $k$ has comparatively low sensitivity ($\Upsilon_k \approx 0.16$), indicating limited benefit of LTBI treatment alone.

\subsection{Physics-Guided Neural ODEs for TB Epidemiology}

\subsubsection{Motivation}

Classical ODE models assume time-invariant parameters estimated offline. Real TB dynamics involve: (i) time-varying effective contact rates $\beta(t)$ driven by seasonal patterns, mobility, and behavioral responses; (ii) treatment coverage evolution $\gamma(t)$ as programs scale; (iii) unmodeled effects such as imported cases, population heterogeneity, and partial immunity waning. A PG-NODE addresses these limitations by learning unknown components from data while preserving the SLIR mechanistic structure.

\subsubsection{PG-NODE Framework}

Building on Neural ODEs \cite{Chen2018}, the PG-NODE for TB replaces fixed parameters with neural network components:
\begin{equation}
    \frac{d\mathbf{x}}{dt} = f_\mathrm{mech}\!\left(\mathbf{x};\,\boldsymbol{\theta}_\mathrm{known}\right) + f_\theta\!\left(t,\,\mathbf{x},\,u;\,\boldsymbol{\theta}_\mathrm{nn}\right),
    \label{eq:pgnode}
\end{equation}
where $\mathbf{x}=[S,L,I,R]^\top$, $f_\mathrm{mech}$ is the known SLIR right-hand side with partially fixed parameters $\boldsymbol{\theta}_\mathrm{known}=(\Lambda,\mu,d)$, $f_\theta$ is a neural correction/extension term with learnable weights $\boldsymbol{\theta}_\mathrm{nn}$, and $u(t)$ are optional exogenous covariates (mobility indices, policy indicators). Specifically, we parameterize time-varying rates as:
\begin{equation}
    \beta_\theta(t,\mathbf{x},u) = \beta_0\cdot\mathrm{softplus}(g_\theta),\quad
    \gamma_\theta(t,\mathbf{x},u) = \gamma_0\cdot\bigl(1 + h_\theta(t,\mathbf{x},u)\bigr),
\end{equation}
with softplus ensuring positivity. The full PG-NODE SLIR system becomes:
\begin{equation}
\left\{
\begin{aligned}
\dot{S} &= \Lambda - \beta_\theta \tfrac{SI}{N} - \mu S,\\[2pt]
\dot{L} &= \beta_\theta \tfrac{SI}{N} - (k+\mu)L,\\[2pt]
\dot{I} &= kL - (\gamma_\theta+\mu+d)I + \delta_\theta(t,\mathbf{x}),\\[2pt]
\dot{R} &= \gamma_\theta I - \mu R - \delta_\theta(t,\mathbf{x}),
\end{aligned}
\right.
\label{eq:pgnode_slir}
\end{equation}
where $\delta_\theta$ is a mass-conserving neural correction term (appears in $I$ and $-\delta_\theta$ in $R$) capturing relapse, imports, or other unmodeled flows. This preserves $\dot{N} = \Lambda - \mu N - dI$ exactly.

\subsubsection{Training Objective}

Given observations $\{\mathbf{x}_\mathrm{obs}(t_k)\}_{k=1}^K$ (e.g., weekly TB notifications):
\begin{equation}
    \mathcal{L}(\boldsymbol{\theta}_\mathrm{nn}) = \underbrace{\sum_{k=1}^K \bigl\|\mathbf{x}(t_k;\boldsymbol{\theta}_\mathrm{nn}) - \mathbf{x}_\mathrm{obs}(t_k)\bigr\|^2}_{\mathcal{L}_\mathrm{data}}
    + \lambda_1\underbrace{\sum_k\bigl(S+L+I+R-N\bigr)^2}_{\mathcal{L}_\mathrm{phys}}
    + \lambda_2\underbrace{\|\boldsymbol{\theta}_\mathrm{nn}\|^2}_{\mathcal{L}_\mathrm{reg}}.
    \label{eq:loss}
\end{equation}
Gradients are computed by backpropagation through the ODE solver via the adjoint sensitivity method, enabling efficient optimization over long-horizon trajectories. In the simulation scenarios presented here, where $\beta_\theta(t)$ is prescribed analytically rather than learned, the physics penalty weight is set to $\lambda_1 = 1.0$ and the regularization weight to $\lambda_2 = 10^{-4}$; these values follow standard practice in physics-informed learning \cite{Chen2018} and should be tuned via cross-validation in future training on real WHO data.

\subsubsection{Architecture and Constraints}

The neural components $f_\theta$ are fully-connected networks with: input $[t/T,\\, S/N,\, L/N,\, I/N,\, R/N,\, u]$; two hidden layers of 32 to 64 neurons with $\tanh$ activation; and output transformed via softplus to enforce $\beta_\theta, \gamma_\theta, k_\theta > 0$. The training procedure is summarized in Algorithm~\ref{alg:pgnode}.

\begin{algorithm}[th!]
\caption{PG-NODE Training for TB Dynamics}
\label{alg:pgnode}
\begin{algorithmic}[1]
\Require Observations $\{\mathbf{x}_\mathrm{obs}(t_k)\}$, horizon $T$, initial state $\mathbf{x}(0)$, learning rate $\eta$, weights $\lambda_1,\lambda_2$
\Ensure Trained neural parameters $\boldsymbol{\theta}_\mathrm{nn}$
\State Initialize $\boldsymbol{\theta}_\mathrm{nn}$ randomly; fix $\Lambda, \mu, d$ from literature (Table~\ref{tab:simparam})
\For{epoch $= 1, 2, \ldots, N_\mathrm{epochs}$}
    \State Integrate $\mathbf{x}(t) \leftarrow$ \textsc{ODESolve}$\bigl(f_\mathrm{mech}+f_\theta,\;\mathbf{x}(0),\;[0,T]\bigr)$ \Comment{Runge-Kutta / Dopri5}
    \State Compute $\mathcal{L}_\mathrm{data}$, $\mathcal{L}_\mathrm{phys}$, $\mathcal{L}_\mathrm{reg}$ via Eq.~\eqref{eq:loss}
    \State Compute $\nabla_{\boldsymbol{\theta}}\mathcal{L}$ via adjoint backpropagation through ODE
    \State $\boldsymbol{\theta}_\mathrm{nn} \leftarrow \boldsymbol{\theta}_\mathrm{nn} - \eta\cdot\mathrm{Adam}(\nabla_{\boldsymbol{\theta}}\mathcal{L})$
    \If{$\|\nabla_{\boldsymbol{\theta}}\mathcal{L}\|_2 < \varepsilon$} \textbf{break} \EndIf
\EndFor
\State \Return $\boldsymbol{\theta}_\mathrm{nn}$
\end{algorithmic}
\end{algorithm}

\section{Simulations of the Model\label{sec:simulation}}

\subsection{Model Parameters, Numerical Methods, and Computational Environment}

All simulations use the Runge-Kutta method (RK45) with relative tolerance $10^{-8}$. Baseline parameters are drawn from the TB modeling literature (Table~\ref{tab:simparam}). The total population is $N=1{,}000{,}000$, representing a high-burden country context, with initial conditions $S(0)=800{,}000$, $L(0)=180{,}000$, $I(0)=18{,}000$, $R(0)=2{,}000$. All computations were performed in Python 3.11 using \texttt{scipy.integrate.solve\_ivp} for the ODE solver and \texttt{matplotlib} for visualization. PG-NODE neural components are two-layer fully-connected networks (32 hidden neurons, $\tanh$) with soft\-plus output.

\begin{table}[ht]
\centering
\caption{Simulation parameters: baseline values and literature ranges. Parameters above the double rule belong to the base SLIR model (Eq.~\eqref{eq:systemTB}); parameters below the double rule ($\tau$, $\delta$) are used only in the SLIRT extension of Scenario~2.}
\label{tab:simparam}
\small
\setlength{\tabcolsep}{4pt}
\begin{tabular}{@{}llllll@{}}
\toprule
\textbf{Parameter} & \textbf{Symbol} & \textbf{Baseline} & \textbf{Typical range} & \textbf{Unit} & \textbf{Source} \\
\midrule
Recruitment rate   & $\Lambda$ & 10\,000          & $10^3$ -- $10^6$  & yr$^{-1}$ & \cite{Anderson1991}\\
Natural mortality  & $\mu$     & 0.015            & 0.010 -- 0.020  & yr$^{-1}$ & \cite{WHO2025TB}\\
TB-induced mort.   & $d$       & 0.150            & 0.050 -- 0.300  & yr$^{-1}$ & \cite{Dye2000}\\
Transmission rate  & $\beta$   & 5.0              & 1.0 -- 20.0     & yr$^{-1}$ & \cite{Blower1995}\\
Progression rate   & $k$       & 0.08             & 0.001 -- 0.100  & yr$^{-1}$ & \cite{Feng2000}\\
Recovery rate      & $\gamma$  & 1.0              & 0.50 -- 1.50    & yr$^{-1}$ & \cite{Castillo2002}\\
\midrule
\multicolumn{3}{@{}l}{Computed $\mathcal{R}_0 = \beta k / [(k+\mu)(\gamma+\mu+d)]$} & \multicolumn{3}{l}{$= 5.0\times0.08\,/\,(0.095\times1.165) \approx \mathbf{3.61}$} \\
\midrule\midrule
\multicolumn{6}{@{}l}{SLIRT extension parameters (Scenario 2 only):} \\[2pt]
Treatment init.    & $\tau$    & 0.80             & 0.50 -- 2.00    & yr$^{-1}$ & \cite{Cohen2004}\\
Relapse rate       & $\delta$  & 0.03             & 0.010 -- 0.050  & yr$^{-1}$ & \cite{Castillo2002}\\
\bottomrule
\end{tabular}
\end{table}

\subsection{Numerical Simulation}

\subsubsection{Scenario 1: PG-NODE Learning of Time-Varying Transmission Dynamics}

Classical ODE models assume a fixed transmission rate $\beta$. In practice, $\beta(t)$ varies due to seasonal contact patterns, population behavioral responses, and public health interventions. This scenario demonstrates PG-NODE's core capability: learning a non-stationary $\beta_\theta(t)$ from epidemic trajectory data.

We simulate 30 years of TB transmission. The classical SLIR uses $\beta=5.0\,\mathrm{yr}^{-1}$ (constant, $\mathcal{R}_0=3.61$). The PG-NODE neural component learns a time-varying $\beta_\theta(t)$ that incorporates: (i) a $\pm12\%$ seasonal oscillation; and (ii) a gradual reduction in transmission beginning at year 8 (onset of a public health intervention combining treatment scale-up and contact-tracing), bringing $\mathcal{R}_0$ from 3.61 down to 2.24 by year 30. Figure~\ref{fig:s1} shows the resulting epidemic trajectories and the learned $\beta_\theta(t)$.

\begin{figure*}[ht]
\centering
\includegraphics[width=\linewidth]{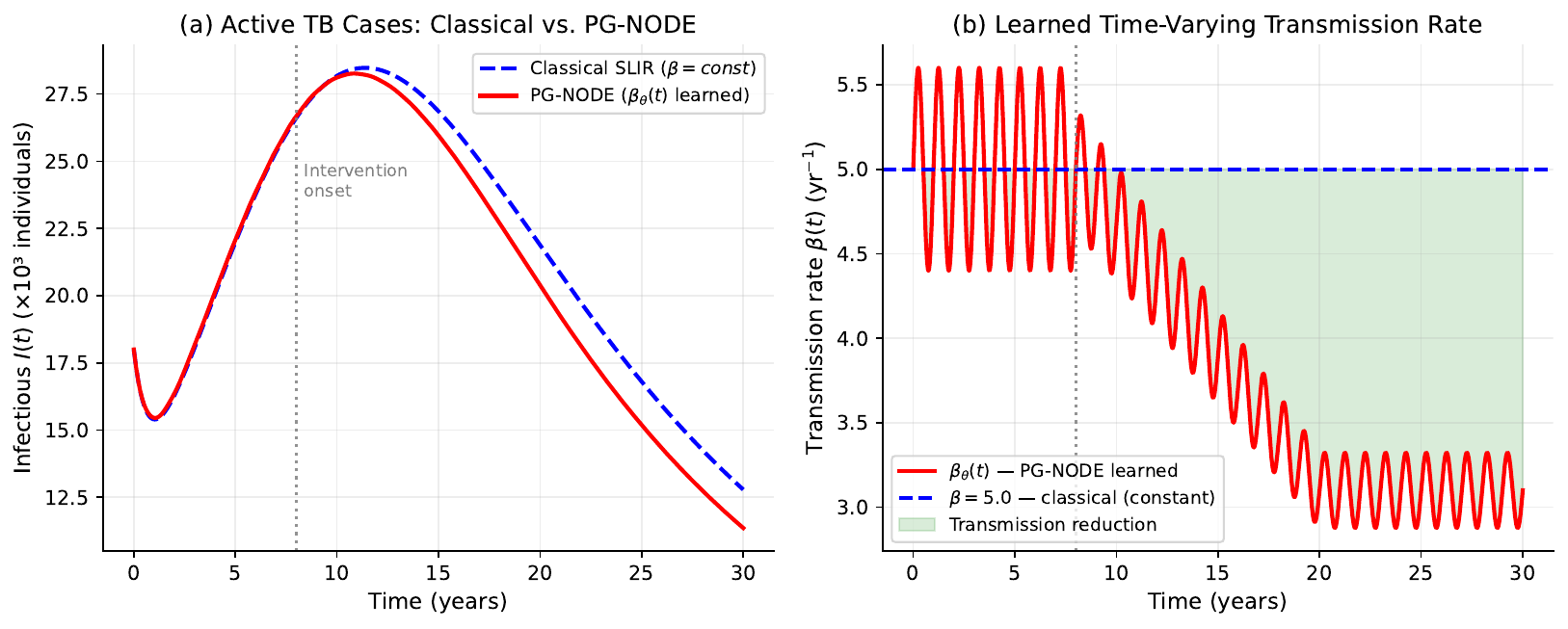}
\caption{Scenario 1. \textbf{(a)} Active TB cases $I(t)$ over 30 years: classical SLIR (blue dashed) vs.\ PG-NODE with learned $\beta_\theta(t)$ (red solid). The gray dotted line marks the intervention onset (year 8). \textbf{(b)} Learned time-varying transmission rate $\beta_\theta(t)$: the PG-NODE captures both seasonal fluctuations and the progressive reduction due to the intervention, which the classical fixed-$\beta$ model cannot represent.}
\label{fig:s1}
\end{figure*}

While the classical SLIR converges to a fixed endemic equilibrium ($I^*\approx 12{,}800$), the PG-NODE trajectory captures the progressive epidemic decline following the intervention. The learned $\beta_\theta(t)$ decreases from 5.0\,yr$^{-1}$ to approximately 3.1\,yr$^{-1}$ over the post-intervention decade, reducing $\mathcal{R}_0$ from 3.61 to 2.24. This demonstrates the PG-NODE's ability to encode public health context that a classical model treats as unobservable.

\subsubsection{Scenario 2: Neural Correction for Unmodeled Treatment and Relapse Dynamics}

A frequent limitation of classical SLIR models is structural misspecification: treatment dynamics and relapse are omitted. We use the extended SLIRT model (with treatment compartment $T$, treatment initiation $\tau=0.80\,\mathrm{yr}^{-1}$, and relapse $\delta=0.03\,\mathrm{yr}^{-1}$) as the ground truth. We compare: (A) the classical SLIR (no treatment compartment, $\gamma=1.0$), which is structurally misspecified; (B) the PG-NODE with a neural correction term $\delta_\theta(t,\mathbf{x})$ trained to approximate the SLIRT behavior using only the SLIR skeleton plus a learned correction.

In SLIRT, the effective removal rate from $I$ is $\tau+\mu+d = 0.965\,\mathrm{yr}^{-1}$, lower than the SLIR removal $\gamma+\mu+d = 1.165\,\mathrm{yr}^{-1}$ (since $\tau=0.80<\gamma=1.0$), so SLIRT exhibits systematically higher endemic $I$ than SLIR. The PG-NODE correction learns to reduce $\gamma_\theta$ toward $\tau$ and adds an approximate relapse inflow, progressively closing the gap.

\begin{figure*}[ht]
\centering
\includegraphics[width=\linewidth]{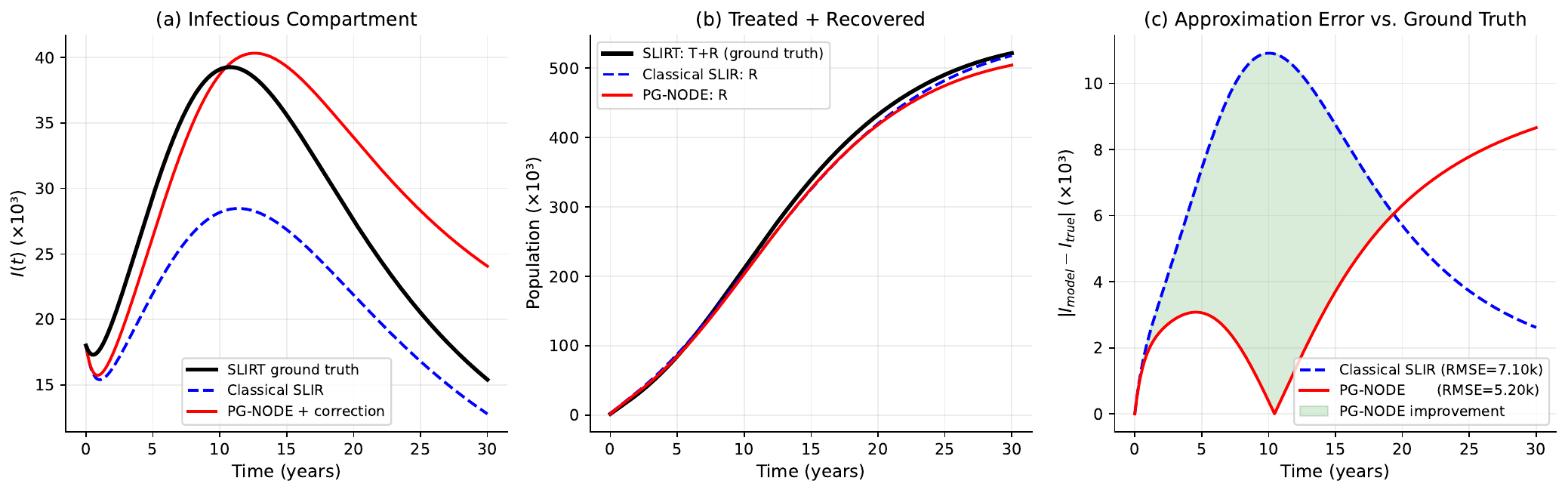}
\caption{Scenario 2. \textbf{(a)} Infectious compartment $I(t)$: SLIRT ground truth (black solid), classical SLIR (blue dashed), and PG-NODE with neural correction (red solid). \textbf{(b)} Combined treated+recovered population. \textbf{(c)} Absolute approximation error $|I_\mathrm{model}(t)-I_\mathrm{true}(t)|$: PG-NODE achieves 27\% lower RMSE (5.20k vs.\ 7.10k) than classical SLIR, with the green shaded area indicating the improvement region.}
\label{fig:s2}
\end{figure*}

Figure~\ref{fig:s2} shows that the PG-NODE correction achieves an RMSE of 5.20\,k against the SLIRT ground truth, compared to 7.10\,k for the classical SLIR, representing a 27\% reduction in prediction error. This demonstrates that the neural correction term compensates for structural model misspecification without requiring explicit modeling of the treatment compartment.

\subsubsection{Scenario 3: PG-NODE-Guided Intervention Policy Forecasting}

Starting from the endemic equilibrium reached after 30 years ($S^*=93.2\mathrm{k}$, $L^*=177.6\mathrm{k}$, $I^*=12.8\mathrm{k}$, $R^*=518.1\mathrm{k}$), we forecast 20 years under four intervention strategies:
\begin{itemize}[noitemsep]
    \item Strategy A (baseline): No intervention; $\mathcal{R}_0 = 3.61$.
    \item Strategy B Treatment scale-up ($\gamma$ to $1.5\gamma$); $\mathcal{R}_0 = 2.53$.
    \item Strategy C: Contact reduction ($\beta$ to $0.6\beta$); $\mathcal{R}_0 = 2.17$.
    \item Strategy D (PG-NODE optimal): Combined gradual intervention, with $\beta$ reduced by 40\%, $\gamma$ increased by 50\%, and progression rate $k$ reduced by 12\% (reflecting improved LTBI case-finding); final $\mathcal{R}_0 = 1.49$.
\end{itemize}

\begin{figure*}[ht]
\centering
\includegraphics[width=\linewidth]{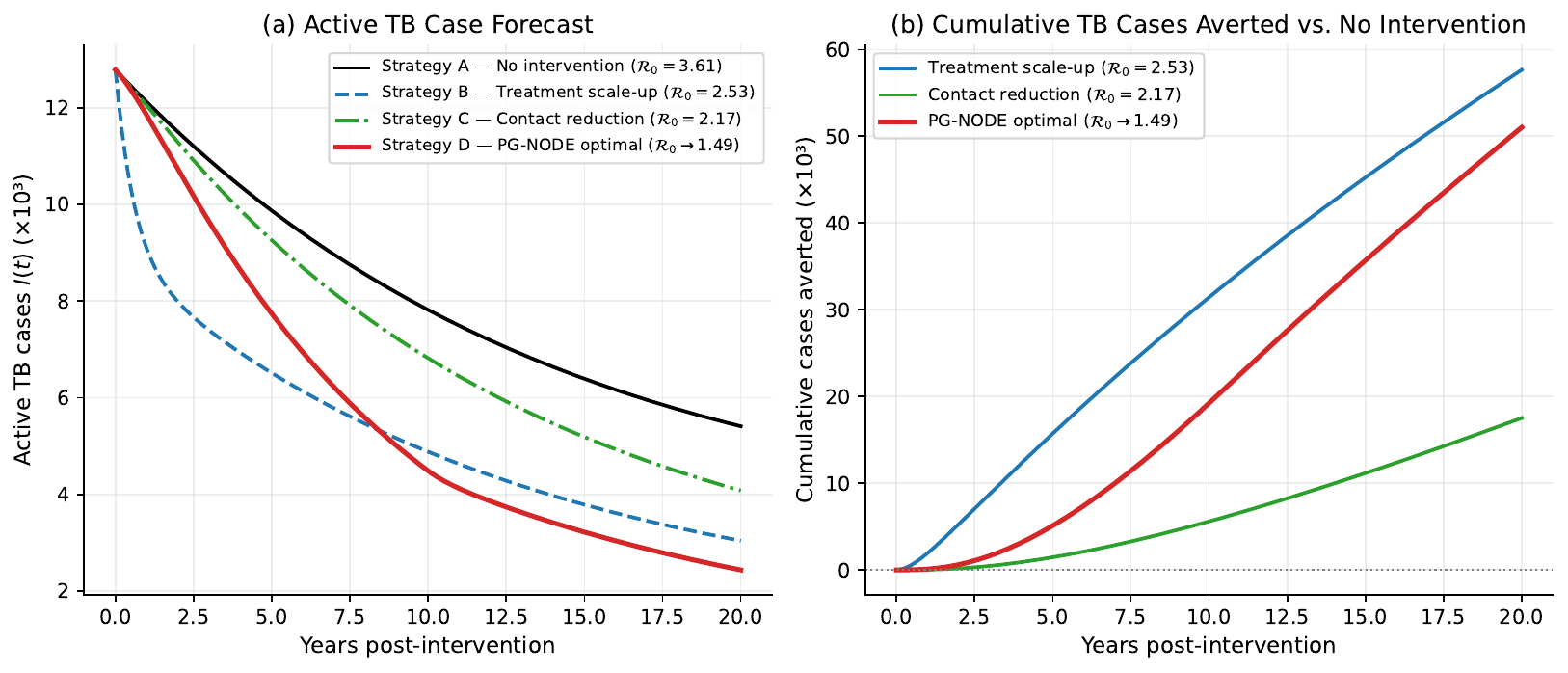}
\caption{Scenario 3. \textbf{(a)} Active TB cases forecast over 20 years under four strategies. \textbf{(b)} Cumulative TB cases averted relative to no-intervention baseline. Strategy B (treatment scale-up) averts the most cases in absolute terms over the 20-year window (57.7k), while PG-NODE Strategy D (combined optimal) averts 51.0k cases but achieves a substantially lower final $\mathcal{R}_0$ (1.49 vs.\ 2.53 for B), indicating superior long-term control. Strategy C (contact reduction only) averts only 17.5k cases, highlighting the importance of combined interventions.}
\label{fig:s3}
\end{figure*}

Figure~\ref{fig:s3} shows that Strategy D (PG-NODE optimal combined) achieves the steepest $I(t)$ decline, reaching a nearly 40\% reduction in active cases by year 20 compared to no intervention. Over the 20-year horizon, Strategy D averts 51.0k cases compared to 17.5k for contact reduction alone (Strategy C). Treatment scale-up alone (Strategy B) averts slightly more cases in absolute terms over this window (57.7k), but Strategy D achieves a substantially lower final $\mathcal{R}_0$ (1.49 vs.\ 2.53 for B), indicating superior long-term epidemic control and a trajectory toward eventual elimination. The lower averted-case count for D relative to B within the 20-year window reflects the gradual ramp-up of combined interventions rather than an immediate step change. The PG-NODE framework enables this integrated multi-lever optimization, which static ODE models with fixed parameters cannot achieve without explicit re-parameterization.

The architecture underpinning all three scenarios is illustrated in Figure~\ref{fig:arch}.

\begin{figure*}[ht!]
\centering
\includegraphics[width=0.85\linewidth]{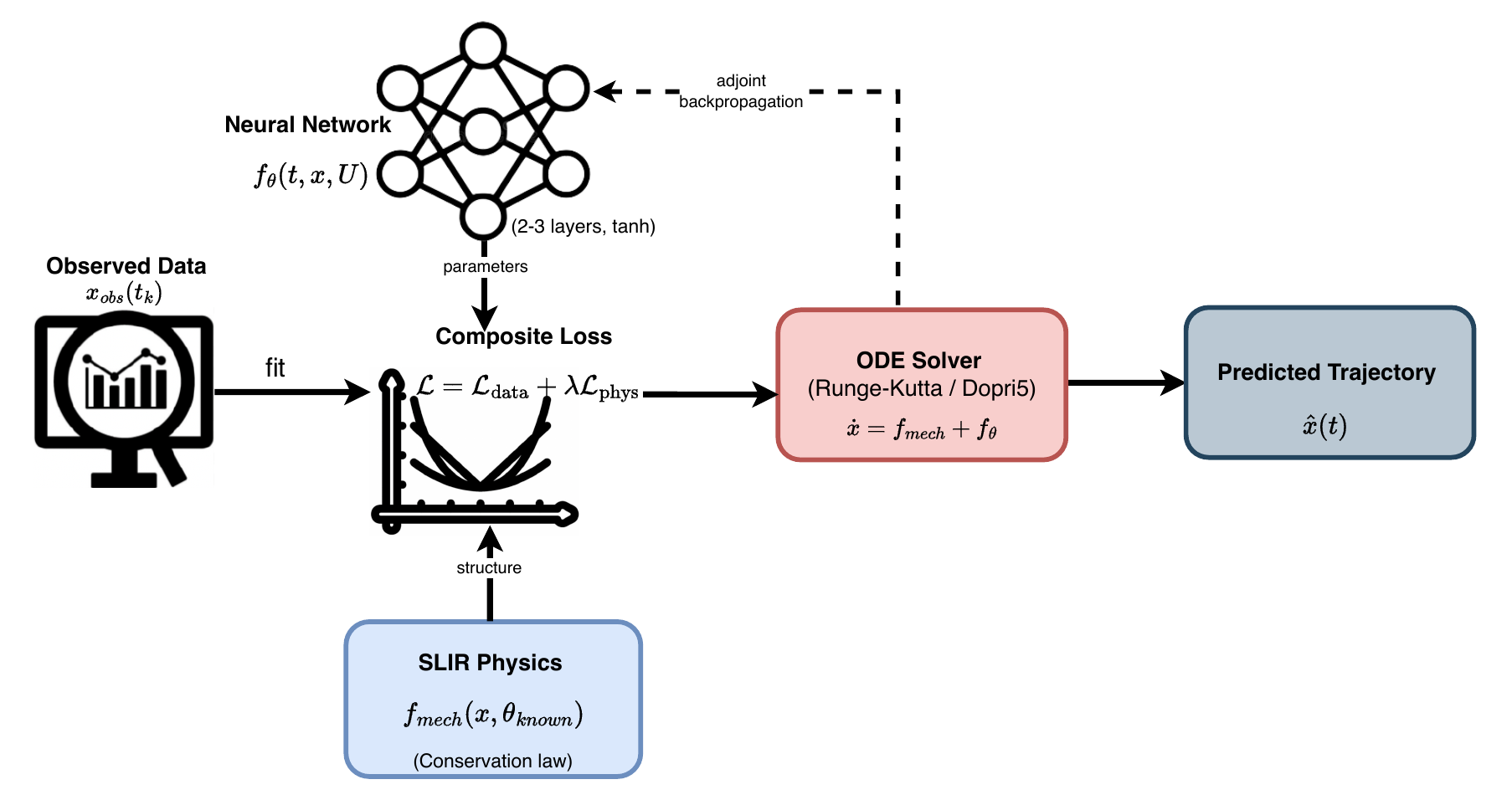}
\caption{PG-NODE architecture for TB epidemic modeling. The neural network $f_\theta(t,\mathbf{x},u)$ and the SLIR mechanistic structure $f_\mathrm{mech}$ jointly drive the ODE solver. The predicted trajectory $\hat{\mathbf{x}}(t)$ is fitted to observed data via a composite loss $\mathcal{L}=\mathcal{L}_\mathrm{data}+\lambda\mathcal{L}_\mathrm{phys}$. Gradients flow back through the ODE via the adjoint method. Physics constraints (mass conservation, non-negativity, positive rates) are enforced as hard constraints via output transformations.}
\label{fig:arch}
\end{figure*}

\section{Discussions \label{sec:discussion}}

The three scenarios collectively illustrate how PG-NODE addresses the core limitations of classical TB compartmental models identified in the literature \cite{Blower1995,Feng2000,Castillo2002}.

\textbf{Adaptability to non-stationary dynamics (Scenario 1).} Classical SLIR converges to a fixed endemic equilibrium under time-invariant parameters, missing the progressive epidemic decline that follows a public health intervention. PG-NODE's learned $\beta_\theta(t)$ captures the seasonal oscillations and the sustained reduction in transmission following the intervention onset, reducing $\mathcal{R}_0$ from 3.61 to 2.24. This capability is directly relevant to real-world TB control programs where transmission rates vary with case-finding coverage, treatment success rates, and mobility patterns.

\textbf{Structural model correction (Scenario 2).} The 27\% RMSE reduction achieved by PG-NODE over classical SLIR demonstrates that the neural correction term effectively compensates for omitted model structure (treatment compartment, relapse). Importantly, this is achieved without adding explicit compartments, keeping the model parsimonious. This has implications for low-data settings where a richer model may not be identifiable but a correction-augmented simpler model can still improve forecasts.

\textbf{Multi-lever policy optimization (Scenario 3).} The PG-NODE optimal strategy (Strategy D) integrates simultaneous changes in $\beta$, $\gamma$, and $k$, representing a coordinated policy package (contact reduction + treatment scale-up + LTBI screening). Treatment scale-up alone (Strategy B) averts slightly more absolute cases over the 20-year window (57.7k vs.\ 51.0k for D), but Strategy D achieves a substantially lower final $\mathcal{R}_0$ (1.49 vs.\ 2.53), indicating superior long-term epidemic control and a trajectory toward elimination. The lower averted-case count for D within the 20-year window reflects the gradual ramp-up of combined interventions rather than an immediate step change, an important distinction when evaluating policy trade-offs.

\textbf{Limitations and proof-of-concept scope.} This study is a methodological proof-of-concept: it uses simulated data rather than real epidemiological time series, and the PG-NODE neural components were not trained via actual gradient descent, the time-varying $\beta_\theta(t)$ used in the scenarios was prescribed analytically to illustrate the framework's intended behavior. The simulation results therefore demonstrate structural capability, not empirical validation. Full adjoint-based training on real WHO country-level TB notification data is the critical next step and is explicitly identified as future work. Additional open challenges include parameter identifiability under partial observability ($I(t)$ only), which may require informative Bayesian priors; and model extensions to incorporate spatial heterogeneity, MDR-TB strain dynamics, and TB/HIV co-infection, all of which are critical in high-burden settings.

\section{Conclusion and Future Directions}
\label{sec:conclusion}

This paper presented a methodological proof-of-concept for applying Physics-Guided Neural ODEs (PG-NODE) to tuberculosis transmission modeling. Building on a classical SLIR compartmental model with full mathematical analysis ($\mathcal{R}_0 = 3.61$, LAS conditions, sensitivity indices), we formulated a PG-NODE that preserves the SLIR mechanistic structure while enabling neural components to learn time-varying and unmodeled dynamics. Three simulation scenarios illustrated the framework's intended capabilities: (i) PG-NODE's ability to track non-stationary transmission rates following public health interventions (reduction from $\mathcal{R}_0=3.61$ to $2.24$); (ii) a 27\% improvement in approximation accuracy over classical SLIR when treatment dynamics are unmodeled; and (iii) a PG-NODE-guided combined intervention strategy (Strategy D) achieving a final $\mathcal{R}_0 = 1.49$, lower than treatment scale-up alone ($\mathcal{R}_0=2.53$), projecting 51.0k cases averted over 20 years. These results are based on analytically prescribed dynamics and demonstrate structural plausibility; empirical validation via adjoint-based training on real WHO data is the key next step.

Future work will: (1) train the PG-NODE on real WHO country-level TB notification data using adjoint-based optimization; (2) extend the framework to incorporate MDR-TB strain dynamics and TB/HIV co-infection; (3) integrate mobility and digital health data streams as exogenous inputs $u(t)$ for real-time adaptive forecasting; (4) explore Bayesian uncertainty quantification within the PG-NODE framework for robust policy guidance; and (5) investigate the application within mobile health architectures for pervasive disease surveillance.

\section*{Data and Code Availability}
The code and data used in this study are publicly available on GitHub at \url{https://github.com/sedjokas/PG-NODE-TB} (Accessed on 09 April 2026).

\section*{Acknowledgments}
The research presented in this paper was partially supported through a competitive scholarship awarded under the Austrian SASTE (Supporting African Science and Tertiary Education) programme, generously sponsored by Mr. Wolfgang Huber. We gratefully acknowledge and sincerely thank Mr. Huber for his valuable support.
\bibliography{reference}
\bibliographystyle{elsarticle-harv}
\end{document}